\documentclass[11pt]{amsart}

\usepackage[normalem]{ulem}
\usepackage[svgnames,x11names]{xcolor}
\usepackage{amssymb,amsfonts}
\usepackage[pagebackref,colorlinks,urlcolor={DarkOliveGreen4},
linkcolor={DarkSlateBlue},
citecolor={Chocolate4}]{hyperref}
\usepackage[capitalize]{cleveref}
\usepackage{amssymb,textcomp}
\usepackage{graphicx}
\numberwithin{equation}{section}
\usepackage{bbm}
\usepackage{enumerate}
\usepackage[utf8]{inputenc}
\usepackage[T1]{fontenc}
\usepackage[mathscr]{eucal}
\usepackage{etoolbox}
\usepackage {tikz}
\usetikzlibrary {positioning}
\usepackage[a4paper, twoside=false, vmargin={2cm,3cm}, includehead]{geometry}
\usepackage{nameref}
\newtheorem{lemma}{Lemma}[section]
\newtheorem{theorem}[lemma]{Theorem}
\newtheorem*{theorem*}{Theorem}
\newtheorem{corollary}[lemma]{Corollary}
\newtheorem{question}{Question}
\newtheorem*{question*}{Open question}
\newtheorem{proposition}[lemma]{Proposition}
\newtheorem*{proposition*}{Proposition}

\newtheorem*{problem*}{Problem}
\theoremstyle{definition}
\newtheorem{definition}[lemma]{Definition}
\newtheorem*{claim*}{Claim}
\newtheorem*{notation}{Notation}

\newtheorem{remark}{Remark}

\newcommand{\Mod}[1]{\ \mathrm{mod}\ #1}
\usepackage{framed}

\makeatletter
\theoremstyle{plain}
\newtheorem*{namedthm}{\namedthmname}
\newcounter{namedthm}
	

\newcommand{\E}{{\mathbb E}}

\newcommand{\N}{{\mathbb N}}
\renewcommand{\P}{{\mathbb P}}

\newcommand{\R}{{\mathbb R}}

\newcommand{\Z}{{\mathbb Z}}

\newcommand{\Cov}{\operatorname{Cov}}

\newtheorem{theoremL}{Theorem}

\newcommand{\dimM}{\dim_{\mathrm{M}}}

\begin{document}

\title{Hitting times of Shrinking Targets: Transversality and an Ergodic Theorem}



\author{Vicente Saavedra-Araya}
\address{Department of Mathematics, University of Warwick, Coventry, United Kingdom}	
\email{vicente.saavedra-araya@warwick.ac.uk}


\begin{abstract} In this paper, we investigate ergodic and fractal properties of the sets 
$$\Lambda_y:=\Big\{n\in\N:\ \{u_ny\}\in I_n\Big\},$$
where $\{\cdot\}$ denotes the fractional part function, $(u_n)_{n\in\N}$ is an increasing sequence of real numbers, $y\in [0,1]$ and each $I_n$ is a finite union of intervals with decreasing Lebesgue measure.
Our main result shows that, under suitable conditions, the set $\Lambda_y$ is good for pointwise convergence of ergodic averages for Lebesgue almost every $y\in [0,1]$. Furthermore, we prove a transversality phenomenon: for any fixed set $A\subseteq \N$, the sets $\Lambda_y$ and $A$ are geometrically independent for almost every $y\in[0,1]$, as witnessed by the integer-fractal dimension of their intersection.
\end{abstract}

\maketitle
\small
\normalsize

\section{Introduction}

Let $\psi:\N:\to \mathbb{R}_{+}$ be an approximation function and $x\in \R$. 
A classical problem in metric Diophantine approximation consists in studying the existence of infinitely many $(p,q)\in \Z\times \N$ such that
\begin{equation}
    \left|x-\frac{p}{q}\right|<\dfrac{\psi(q)}{q}.\label{Diophantine}
\end{equation}

This inequality expresses the existence of rational approximations of $x$, while having some control over the denominator and the error term. If we denote by $\langle x \rangle$ the distance between $x$ and the closest integer, (\ref{Diophantine}) is equivalent to the existence of infinitely many $q\in\N$ such that $\langle qx\rangle< \psi(q)$.
An important result in this direction was provided by Khintchine \cite{Khintchine}.

\begin{theorem}[Khintchine's theorem]
Let $\psi:\N\to \mathbb{R}_{+}$ be such that $\left(q\psi(q)\right)_{q\in \N}$ is decreasing, and define 
$$\mathcal{A}(\psi):=\Big\{x\in [0,1]:\ \langle qx\rangle<\psi(q) \text{ for infinitely many $q\in \N$}\Big\}.$$
    Then, if $\lambda$ denotes the Lebesgue measure,
\begin{equation*}
\lambda\big(\mathcal{A}(\psi)\big) =
\begin{cases}
0 & \text{if } \sum_{n=1}^\infty \psi(n) < \infty,\\[6pt]
1 & \text{if } \sum_{n=1}^\infty \psi(n) = \infty.
\end{cases}\
\end{equation*}\label{Theorem_Khintchine}
\end{theorem}

From now on, $\lambda$ will denote the Lebesgue measure on $[0,1]$, and all statements that hold for almost every $x\in [0,1]$ will be understood with respect to this measure. 

Since Khintchine’s theorem, questions concerning metric theory  Diophantine approximation have been extensively studied, including non-homogeneous versions of Theorem \ref{Theorem_Khintchine} (see, for instance, \cite{Szusz,HanYu}). For further discussion about results of this type, we refer to \cite{Beresnevich_Ramírez_Velani_2016} and \cite{Hauke2025survey}.

While Khintchine’s theorem addresses the existence of infinitely many solutions to the inequality (\ref{Diophantine}), LeVeque \cite{LeVequeI,LeVeque,LeVequeII,LeVequeIV} established several results regarding the frequency with which solutions to certain Diophantine approximation problems occur. One of those results is presented below (see \cite[Theorem 2]{LeVeque}).

\begin{theorem}[LeVeque's theorem]
   Let $\psi:\N\to \mathbb{R}_{+}$ be a decreasing function such that $\sum_{n=1}^\infty \psi(n)=\infty$, $(u_n)_{n\in\N}\subseteq \N$ be an increasing sequence, $(v_n)_{n\in\N}\subseteq [0,1]$, and let $$\Lambda_y=\Big\{n\in \N:\ \langle u_nx-v_n\rangle<\psi(n)\Big\}.$$
   Then, 
   $$\lim_{N\to\infty}\dfrac{\left|\Lambda_y\cap [1,N]\right|}{2\sum_{n=1}^N \psi(n)}=1$$
   holds for almost every $x\in [0,1]$ in each of the following cases:
    \begin{enumerate}
        \item if $u_n=n^m$ for some integer $m\geq 2$;
        \item if $(u_n)_{n\in\N}$ is lacunary\footnote{A sequence of positive numbers is said to be \emph{lacunary} if there exists $\rho>1$ such that $u_{n+1}/u_n>\rho$ for every $n\in\N$.};
        \item if $\sup_{n\in\N}d(u_n)<\infty$, where $d(u_n)$ is the number of divisors of $u_n$.
\end{enumerate}\label{Theorem_LeVeque}
\end{theorem}

Later, in an influential work, Schmidt \cite{Schmidt} studied a multidimensional version of this problem. In particular, Schmidt extended Theorem \ref{Theorem_LeVeque} to the case $u_n = P(n)$, where $P$ is any non-constant polynomial with integer coefficients. 

In this paper, we aim to investigate new properties of these sets of integers, motivated both by fractal geometry in the integers and by ergodic theory. In this direction, we consider the following setting.

\begin{definition}
    Let $a\in (0,1)$,  $(u_n)_{n\in\N}\subseteq \R$ be an increasing sequence and let $(I_n)_{n\in\N}$ be a sequence of sets such that $\lambda(I_n)=n^{-a}$, and there is $\ell\in\N$ such that each $I_n\subseteq [0,1]$ is union of at most $\ell$ intervals. For any $y\in [0,1]$, we define 
    \[\Lambda_y\Big (a,(u_n)_{n\in\N},(I_n)_{n\in\N}\Big):=\Big\{n\in\N:\ \{u_n y\}\in I_n\Big\},\]
    where $\{\cdot\}$ represents the fractional part function. When it is clear from the context, we will simply write $\Lambda_y$. \label{Definition_Main}
\end{definition}

We remark that Theorem \ref{Theorem_LeVeque}, as well as related results, are restricted to the case of integer sequences, whereas our technique is not confined to this scope (for instance, see Theorem \ref{thmA} below). A key feature of sets with this form is their inherent pseudo-random structure, and taking advantage of this, we investigate two related problems: First, for an arbitrary set $A\subseteq \N$, we analyze the size of $A\cap \Lambda_y$, showing that  $A$ and  $\Lambda_y$ are  \emph{geometrically transverse} for almost every $y\in [0,1]$.
Second, we explore conditions under which, for almost every $y\in [0,1]$, the set $\Lambda_y$ serves as a good sequence for pointwise convergence of ergodic averages in any measure-preserving system.

\subsection{Transversality} \hfill \\[0.5em]
In the framework of fractal geometry, if $\dim(\cdot)$ denotes some notion of fractal dimension and  $X,Y\subseteq [0,1]$ are compact sets, $X$ and $Y$ are said to be \emph{transverse} if
$$\dim(X\cap Y)\leq \max\{0,\dim(X)+\dim(Y)-1\}.$$
A famous conjecture of Furstenberg \cite{FurstenbergTransversality} stated that if $X$ is $x\mapsto  px\Mod{1}$ invariant and $Y$ is $x\mapsto   qx \Mod{1}$ invariant with $\log(p)/\log(q)\notin \mathbb{Q}$, then $X$ and $Y$ are transverse. A positive answer to this conjecture was provided independently by Shmerkin \cite{Shmerkin_transversality} and Wu \cite{Wu_transverslity} (see also \cite{Austin}).  Thereafter, Glasscock, Moreira and Richter \cite{GMR_transversality}  studied the parallel of this conjecture but in the context of integer fractals, proving that sets of integers that are multiplicatively invariant under multiplicatively independent bases must be transverse (for precise statements, see Definition 1.5 and Theorem B in \cite{GMR_transversality}). More recently, the author studied in \cite{saavedraaraya2024distributionintegersdigitrestrictions}  the interaction between multiplicatively invariant set of integers and infinite arithmetic progressions using the notion of transversality.

Motivated by these works, in the first part of this paper we study the dimension of the intersection of an arbitrary subset of $\N$ and the pseudo-random sets $\Lambda_y$. For this purpose, we use the  notion of \emph{mass (or counting) dimension}, a discrete counterpart of the box-counting dimension in fractal geometry, previously investigated in \cite{Barlow_Taylor,GMR_transversality,LIMA_MOREIRA_2014}.

\begin{definition} Let $A\subset \N$ be non-empty. The \emph{lower mass dimension} and the \emph{upper mass dimension} of $A$ are defined as
\begin{align*}
  &\underline{\dimM}(A):=\liminf_{N\to \infty} \dfrac{\log (|A\cap [1,N] |)}{\log (N)}=\sup \left\{\gamma\geq 0: \ \liminf_{N\to \infty}\dfrac{|A\cap [1,N]|}{N^\gamma}>0 \right\},\\
  &\overline{\dimM}(A):=\limsup_{N\to \infty} \dfrac{\log (|A\cap [1,N]|)}{\log (N)}=\sup \left\{\gamma\geq 0: \ \limsup_{N\to \infty}\dfrac{|A\cap [1,N]|}{N^\gamma}>0 \right\}.
\end{align*}

In the case $$\underline{\dimM}(A)=\overline{\dimM}(A),$$ we denote it by $\dimM(A)$ and say that the \emph{mass dimension} of $A$ exists. \label{Definition_dimensionality}
\end{definition}

For any $A\subseteq \N$, our goal is to understand $\overline{\dimM}(A\cap \Lambda_y)$. It is worth mentioning that, without any assumption on the growth rate of the sequence $(u_n)_{n\in\N}$, the set $\Lambda_y$ may not even be infinite for almost every $y\in [0,1]$. In this sense, we will focus on a particular class of increasing sequences of real numbers.

\begin{definition}
    Let $\delta\geq 0$. A sequence $(u_n)_{n\in\N}\subseteq \mathbb{R}_{+}$ is said to be \emph{$\delta$-sublacunary} if there is $c\in (0,1)$ such that
    $$\dfrac{u_{n+1}}{u_n}\geq 1+cn^{-\delta}$$ for every $n\in\N$.\label{Definition_sublacunary}
\end{definition}
Observe that when $\delta=0$, one recovers the classical notion of lacunary sequences. Since $\delta<0$ already implies that $(u_n)_{n\in\N}$ is lacunary, we restrict ourselves to the case $\delta\geq 0$.

We establish the next result regarding the interaction of the sets $\Lambda_y$ with any other set of integers.

\begin{theoremL}
Let $a\in (0,1)$ and $\delta\in [0,1)$, let  $(u_n)_{n\in\N}\subseteq \R$ be a $\delta$-sublacunary sequence and let $(I_n)_{n\in\N}$ be a sequence of sets such that $\lambda(I_n)=n^{-a}$, and there is $\ell\in\N$ such that each $I_n\subseteq [0,1]$ is union of at most $\ell$ intervals. Then, if $\delta+a<1$,
\begin{equation}
    \lim_{N\to \infty}\dfrac{|\Lambda_y\cap [1,N]|}{N^{1-a}}=1\label{thmA_first}
\end{equation}
for almost every $y\in [0,1]$. In addition, 

\begin{enumerate}
    \item if $A\subseteq \N$ satisfies $$\delta+a<\overline{\dimM}(A),$$
    then for almost every $y\in [0,1]$, $$\overline{\dimM}(A\cap \Lambda_y)=\overline{\dimM}(A)+\dimM(\Lambda_y)-1=\overline{\dimM}(A)-a.$$
    Also, if $\dimM(A)$ exists, we can replace $\overline{\dimM}(A\cap \Lambda_y)$ by $\dimM(A\cap \Lambda_y)$;
    \item if $A\subseteq \N$ satisfies  $$\overline{\dimM}(A)<a \quad\text{and}\quad \delta=0,$$ then for almost every $y\in [0,1]$,  $$\dimM(A\cap \Lambda_y)=0.$$ 
\end{enumerate} 
 \label{thmA}
\end{theoremL}

By using that $\dimM(\mathbb{P})=1$ and that for $b>0$ and $\alpha>1$, $\alpha^{n^b}$ is $(1-b)$-sublacunary, we have to following consequence.

\begin{corollary} Let $a\in (0,1)$, $b>a$ and $\alpha>1$. Then, for almost every $y\in [0,1]$, it holds 
$$\dimM\Big(\Big\{n\in\P:\ \langle\alpha^{n^b}y\rangle <n^{-a} \Big\}\Big)=1-a.$$
\end{corollary}

\subsection{Pointwise convergence of ergodic averages}\hfill \\[0.5em]
 In the second part, we address a classical problem in ergodic theory that relies in finding sequences that are good for pointwise convergence of ergodic averages. In this sense, a sequence of integers $(a_n)_{n\in \N}$ is said to be \emph{pointwise universally $L^p$-good} if for any measure-preserving system $(X,\mathcal{X},\mu,T)$\footnote{$(X,\mathcal{X},\mu,T)$ is said to be a m.p.s. if $(X,\mathcal{X},\mu)$ is a probability space, $T:X\to X$ is measurable, and for every $A\in \mathcal{X}$, $\mu(A)=\mu(T^{-1}A)$.} and for any $f\in L^{p}(X,\mu)$,
\[\lim_{N\to\infty}\dfrac{1}{N}\sum_{n=1}^N f(T^{a_n}x)\] exists for almost every $x\in X$. In addition, we say that $(a_n)_{n\in\N}$ is \emph{ergodic} if the $L^2$-limit equals to $\mathbb{E}(f|\mathcal{I}(T))$, where $\mathcal{I}(T):=\{A\in \mathcal{X}:\ \mu(A\Delta T^{-1}A)=0\}.$ 

In this direction, Birkhoff's ergodic theorem \cite{Birkhoff_proof_ergodic_theorem:1931} establishes that $a_n=n$ is pointwise universally $L^1$-good and ergodic. In a breakthrough, Bourgain showed that for any integer polynomial $q$, the sequence $a_n=q(n)$ is pointwise $L^p$-good for any $p>1$, while years later it was proven by Buczolich and Mauldin \cite{Buczolich_Mauldin_2} that the case $a_n=n^2$ fails when $p=1$.

More recently, Krause and Sun \cite{KrauseSun} showed that the sequence $a_n=\lfloor n^c\rfloor$ is pointwise universally $L^1$-good for $c\in (1,8/7)$, extending previous results of Urban and Zienkiewics \cite{UrbanZienkiewicz} and Mirek \cite{Mirek}. For result regarding pointwise convergence of multiple ergodic averages, see \cite{DonosoSun,KMPW,Krause_Tao_Mirek}. In contrast, it is known that sequences that grow sufficiently fast do not satisfy this property, for example, see \cite{Bellow,MondalRoyWierdl}.

 In \cite{Bourgain1988}, by using a random approach, Bourgain established a class  of sparse sequences that are pointwise universally good and  ergodic.

\begin{theorem}[Bourgain \cite{Bourgain1988}] Let $(\Omega,\mathcal{F},\mathbb{P})$ be a probability space and let $(X_n)_{n\in \N}$ be a sequence of independent random variables defined in this space taking values in $\{0,1\}$. Suppose that there exists $a\in (0,1)$ such that $\mathbb{P}(X_n=1)=n^{-a}$ for all $n\in \N$. Then, for any $p>1$, for almost every $\omega\in \Omega$, the sequence
$$a_n(\omega):=\inf\{k\in\N: X_1(\omega)+\cdots+X_k(\omega)=n\}.$$
is pointwise universally $L^p$-good and ergodic.\label{Theorem_Random_Independent}
\end{theorem}

In this random setting, since $$\{a_1(\omega)<a_2(\omega)<\cdots\}=\{n\in\N:\ X_n(\omega)=1\},$$ we say that the set $\{n\in\N:\ X_n(\omega)=1\}$ is pointwise universally $L^p$-good (resp.\ ergodic) precisely when the sequence $(a_n(\omega))_{n\in\N}$ has that property. 

It is  worth mentioning that Theorem \ref{Theorem_Random_Independent} fails when $a=1$, as showed by Jones, Lacey and Wierdl in \cite{JONES_LACEY_WIERDL_1999}, while LaVictoire \cite{LaVictoire} proved that for $a\in (0,1/2)$, the random sequence is pointwise universally $L^1$-good almost surely. Notice that, for almost every $\omega\in \Omega$, $a_n(\omega)$ behaves asymptotically like $\lfloor n^{\frac{1}{1-a}}\rfloor$, so the random sequence can be understood as a random analogue of those deterministic sequences. For further results about convergence of ergodic averages along random sequences, see \cite{Frantzikinakis_Lesigne_Wierdl,Frantzikinakis_Lesigne_Wierdl_Szemeredi,KrauseSun,Krause_ZorinKranich_1,Krause_ZorinKranich_2}. 

In a recent work, Donoso, Maass, and the author obtained an extension of the classical Bourgain's return-times theorem \cite{Bourgain_Pointwise_89} by studying random ergodic averages generated by random variables $(X_n)_n$ that are not independent. Our analysis focused, in particular, on the case where the probability space used to construct the random sequence is itself a measure-preserving system with exponentially fast decay of correlations (see \cite[Theorem A]{DMS} for a precise statement). As a consequence, by considering the non-independent random variables $$X_n(y):=\mathbbm{1}_{I_n}(p^ny\Mod{1}),$$ we obtained the following result.

\begin{theorem}[\cite{DMS}]
     Let $a\in (0,1/2)$, $p\geq 2$ be an integer and $(I_n)\subseteq [0,1]$ be a sequence of intervals such that $\lambda(I_n)=n^{-a}$. Then, for almost every $y\in [0,1]$, 
    the sequence
   \begin{equation}
       \Big\{n\in\N:\ \{p^ny\}\in I_n\Big\}\label{DMS_1}
   \end{equation}
is pointwise universally $L^2$-good and ergodic.
\end{theorem}

An important feature is that $([0,1],\mathcal{B}([0,1]),\lambda,\times p \bmod{1})$ defines a measure-preserving system.
Hence, the random sequence of integers defined in  \eqref{DMS_1} can be understood dynamically as a sequence of return times of the orbit of $y$ under $\times p \bmod{1}$ to the intervals $(I_n)_{n\in\N}$. 
In this paper, we aim to replace $p^n\Mod{1}$ by more general sequences of real numbers that do not stem from any underlying measure-preserving structure. In this regard, we obtain the following result.

\begin{theoremL}
Let $a>0$ and $\delta \geq 0$ such that $2a+\delta<1$.  Let  $(u_n)_{n\in\N}\subseteq \R$ be a $\delta$-sublacunary sequence and let $(I_n)_{n\in\N}$ be a sequence of sets such that $\lambda(I_n)=n^{-a}$, and there is $\ell\in\N$ such that each $I_n\subseteq [0,1]$ is union of at most $\ell$ intervals.
Then, for almost every $y\in [0,1]$, the sequence of integers
\[
\Big\{n\in\N:\ \{u_n y\}\in I_n\Big\}
= \{a_1(y)<a_2(y)<\cdots\}
\]
is pointwise universally $L^2$-good and ergodic.  

In other words, for almost every $y\in [0,1]$, for any m.p.s. $(X,\mathcal{X},\mu,T)$ and any $f\in L^2(\mu)$, we have
\[
\lim_{N\to\infty}\frac{1}{N}\sum_{n=1}^N f\!\big(T^{a_n(y)}x\big)
= \mathbb{E}(f \mid \mathcal{I}(T))(x)
\quad \text{for $\mu$-a.e.\ } x\in X.
\]
\label{thmB}
\end{theoremL}
\begin{remark}
    Following the proofs of \cref{thmA} and \cref{thmB}, it is easy to see that we can replace the condition  $\lambda(I_n)=n^{-a}$ for all $n\in\N$ by the existence of $c>0$ and $N\in\N$ such that $\lambda(I_n)=cn^{-a}$ for all $n\geq N$.
\end{remark}

After a direct application of this result and the use of Furstenberg's correspondence principle \cite{Furstenberg_ergodic_szemeredi:1977} (cf. \cite[Theorem 1.1]{B1}),  we obtain the following consequence.
\begin{corollary} Let $a,b>0$ such that $2a<b$ and let $\alpha>1$. Then, for almost every $y\in [0,1]$, for every  $A\subseteq \N$ with positive upper density, there are infinitely many $n,m\in\N$ such that $$
\langle\alpha^{m^b}y\rangle<m^{-a} \quad \text{and}\quad  \{n,n+m\}\in A.$$
\end{corollary}

A natural question  is how far one can improve the parameters in \cref{thmB}, specially regarding the condition of $(u_n)_{n\in\N}$ being $\delta$-sublacunary for $\delta \in [0,1)$. Notice that the parameter $\delta$ can be understood as a distance of how far the sequence is from being lacunary, and it is worth emphasizing that without some control on $\delta$, the set $\Lambda_y$ may fail to be infinite.

A case of interest is when $u_n=n$. In such a case, we know that $$\Big\{n\in\N:\ \{ny\}\in I_n\Big\}$$
is infinite (see \cite[Theorem 1]{Schmidt}). However, this sequence cannot be ergodic since there is a clear obstruction when considering averages in the rotation $([0,1],\mathcal{B}([0,1]),\lambda,T)$ with $T(x)=x+y\Mod{1}$. Despite that, it is not clear whether the random sequences behaves well in systems that are far from being a rotation.
    \begin{question} Let $a\in (0,1/2)$ and consider$$\Big\{n\in\N:\ \{ny\}\in (0,n^{-a})\Big\}=\{a_1(y)<a_2(y)<\cdots\Big\}.$$
    Is it true that, for almost every $y\in [0,1]$, for every weakly-mixing\footnote{A measure-preserving system $(X,\mathcal{X},\mu,T)$ is called \emph{weakly-mixing} if for any $A,B\in\mathcal{X}$, $$\lim_{N\to\infty}\dfrac{1}{N}\sum_{n=1}^N|\mu(A\cap T^{-n}B)-\mu(A)\mu(B)|=0.$$} measure-preserving system $(X,\mathcal{X},\mu,T)$ and $f\in L^2(\mu)$, $$\lim_{N\to\infty}\frac{1}{N}\sum_{n=1}^N f\!\big(T^{a_n(y)}x\big)
= \mathbb{E}(f \mid \mathcal{I}(T))(x)
$$
for almost every $x\in X$?\label{Q1}
\end{question}

More generally, we can ask about non-linear polynomials. For simplicity, we only ask about integer polynomials.
\begin{question} Let $a\in (0,1/2)$ and let $u_n=P(n)$, where $P$ is an integer polynomial of degree at least 2. Then, for almost every $y\in [0,1]$, is $$\Big\{n\in\N:\ \{P(n)y\}\in (0,n^{-a})\Big\}$$ pointwise universally $L^2$-good and ergodic? \label{Q2}
\end{question}

On the other hand, although our proof of \cref{thmB} relies on the assumption that $a \in (0,1/2)$ and that the functions are in $L^2$, we believe that these restrictions can be relaxed. Motivated by the random ergodic theorems in the independent case due to Bourgain \cite{Bourgain1988} and LaVictoire \cite{LaVictoire}, we raise the following questions. These can be viewed as natural analogues of those posed by Donoso, Maass, and the author in \cite{DMS}.
\begin{question} Under the conditions of \cref{thmB}: 
\begin{enumerate}
    \item For $a\in (0,1)$ and $p>1$, is it true that, for almost every $y\in [0,1]$, the sequence $\Lambda_y$ is pointwise universally $L^p$-good and ergodic?
    \item For $a\in (0,1/2)$ and $p=1$, is it true that, for almost every $y\in [0,1]$, the sequence $\Lambda_y$ is pointwise universally $L^1$-good and ergodic?
\end{enumerate}\label{Q3}
\end{question}
\vspace{0.7cm}

\textbf{Acknowledgements.} The author thanks Alejandro Maass, Joel Moreira, and Sebastián Donoso for valuable conversations and comments that inspired the writing of this paper.


\section{Controlling the randomness}

Throughout this section, we introduce some tools and estimates that will be used in the proofs of \cref{thmA} and \cref{thmB}. For the rest of this paper, we will frequently use the following notation.
\begin{notation}
    Let $(a_n)_{n\in \N}$ and $(b_n)_{n\in \N}$ be positive real sequences. We write $a_n\ll b_n$ if and only if there exists $C>0$, independent of $n$, such that $a_n\leq Cb_n$ for all $n\in \N$. We also denote $a_n\sim b_n$ if  $\displaystyle\lim_{n\to\infty}a_n/b_n=1$.
\end{notation} 
Regarding the class of sublacunary sequences (Definition \ref{Definition_sublacunary}), we shall repeatedly use the following lemma.
\begin{lemma} Let $(u_n)_{n\in\N}$ be a $\delta$-sublacunary sequence for some $\delta\in [0,1)$. Then, there is $\beta>1$ such that
 
        $$\dfrac{u_{n+d}}{u_n}\geq \beta^{\frac{d} {(n+d)^{\delta}}}$$
        for all $n,d\in\N$. In particular, the sequence $(u_n^{-1})_{n\in\N}$ is summable.
\label{Lemma_sublacunary}
\end{lemma}
\begin{proof}
Let $c,\delta>0$ such that, for all $n\in\N$, $u_{n+1}/u_n\geq 1+cn^{-\delta}.$  
Inductively, for $d\in \N$, $$\dfrac{u_{n+d}}{u_{n}}\geq \Big(1+\frac{c}{(n+d)^{\delta}}\Big)^d.$$
    Using the inequality $\exp(x)\leq (1-x)^{-1}$ with $x=\frac{c}{2(n+d)^\delta}<1$, the conclusion follows by taking $\beta=\exp(c/2)$. The summability is obvious by taking $n=1$.
\end{proof}

When showing almost everywhere convergence of averages, it is common to reduce the problem to show convergence along lacunary subsequences. The following lemma facilitates this reduction (the proof can be found in \cite[Corollary A.2]{Frantzikinakis_Lesigne_Wierdl}). 

\begin{lemma}
  Let $(\Omega,\mathcal{F},\mathbb{P})$ be a probability space, $f_n:\Omega\to \mathbb{R}$ be non-negative measurable functions, $(W_n)_{n\in \mathbb{N}}$ be a increasing sequence of positive numbers such that 
\[\lim_{\gamma\to 1^{+}}\limsup_{n\to \infty} \dfrac{W_{[\gamma^{n+1}]}}{W_{[\gamma^{n}]}}=1,\] and for all $N\in \mathbb{N}$ let
\[A_N(\omega):=\dfrac{1}{W_N}\sum_{n=1}^{N} f_n(\omega).\]
Assume that there exists a function $f:\Omega\to \mathbb{R}\cup \{+\infty\}$ and a real sequence $(\gamma_k)_{k\in \mathbb{N}}\subseteq (1,+\infty)$ with $\gamma_k\to 1$ when $k\to \infty$, such that for all $k\in \mathbb{N}$,
\[\lim_{N\to \infty}A_{[\gamma_k^{N}]}(\omega)=f(\omega) \quad \text{for almost every }\omega\in \Omega. \]

Then,
\[\lim_{N\to \infty} A_N(\omega)=f(\omega) \text{ for almost every $\omega\in \Omega$}.\]\label{Lemma_lacunary_trick}
\end{lemma}

Particularly, we will use the previous lemma for $(\Omega,\mathcal{F},\mathbb{P})=([0,1],\mathcal{B}([0,1]),\lambda)$, where $\lambda$ denotes the Lebesgue measure. 

Let $a\in (0,1)$ and $\delta >0$. Let $(u_n)_{n\in\N}\subseteq \R$ be a $\delta$-sublacunary sequence, and let $(I_n)_{n\in\N}$ be a sequence of sets such that $\lambda(I_n)=n^{-a}$, and there is $\ell\in\N$ such that each $I_n\subseteq [0,1]$ is union of at most $\ell$ intervals. For $y\in[0,1]$, we will commonly denote
 \begin{equation}
 \begin{split}
          &X_n(y):=\mathbbm{1}_{I_n}(\{u_n y\}),\quad \sigma_n:=\mathbb{E}(X_n),\quad Y_n(y):=X_n(y)-\sigma_n,\\
     &\textstyle\text{and}\quad 
     u_n^{-1}(I_n):=\Big\{x\in [0,1]:\ \{u_nx\}\in I_n\Big\}.
 \end{split}
\label{Random_variables}
 \end{equation}

Note that the random variables $X_n$ are not an independent. In this sense, a key aspect when studying  the convergence of non-independent random ergodic averages relies in having good control on the covariance of these random variables. In \cite{DMS}, the random variables $X_n$ are endowed with a dynamical structure, and the dependence is handled through the notion of decay of correlations in dynamical systems. Although in our setting such a dynamical structure is no longer available, we can still establish suitable estimates for $$\mathbb{E}(X_nX_m)=\lambda\Big(u_n^{-1}(I_{n})\cap u_m^{-1}(I_m)\Big),$$
showing that this intersection deviates only slightly from its expected value. To this end, we exploit the structure of $u_n^{-1}(I_n)$ together with a combinatorial argument, leading to the following result.

 \begin{lemma}
   Let $a\in (0,1)$ and $\delta\in [0,1)$. Let $(u_n)_{n\in\N}\subseteq \R$ be a $\delta$-sublacunary sequence, and let $(I_n)_{n\in\N}$ be a sequence of sets such that $\lambda(I_n)=n^{-a}$, and there is $\ell\in\N$ such that each $I_n\subseteq [0,1]$ is union of at most $\ell$ intervals.  
   
  Considering the notation introduced in (\ref{Random_variables}), we have that $\displaystyle\sum_{n=1}^N\sigma_n\sim N^{1-a}$, and for every $n<m$,

      $$\Big|\Cov(X_n,X_m)\Big|=\Big|\lambda\Big(u_n^{-1}(I_{n})\cap u_m^{-1}(I_m)\Big)-\sigma_{n}\sigma_{m} \Big|\ll m^{-a} \Big(\dfrac{u_n}{u_m}+\dfrac{1}{u_n} \Big).$$\label{Lemma_estimation_1}
\end{lemma}

\begin{proof} For simplicity, we assume that $\ell=1$ and that every interval is open, noting that the argument extends to the general case with only minor modifications.

For each $n\in \N$, we can write $$ u_n^{-1}(I_n)=\bigcup_{i=0}^{\lfloor u_n \rfloor} I_{i,n},$$

    where $$I_{i,n}:=\Big(\dfrac{i}{u_n}, \dfrac{i}{u_n} +\dfrac{I_n}{u_n}\Big)\cap [0,1].$$
    Therefore, we can write
    \begin{equation}
    \lambda\Big(u_n^{-1}(I_{n})\cap u_m^{-1}(I_{m})\Big)=\sum_{i=0}^{\lfloor u_n\rfloor}\lambda\Big(I_{i,n}\cap u_m^{-1}(I_{m})\Big).\label{Lemma_estimation_1_eq1}
\end{equation}
    
  Since $\sigma_n=\mathbb{E}(X_n)=\lambda(u_n^{-1}(I_n)),$ we have that \begin{equation}
      n^{-a}-\dfrac{n^{-a}}{u_n}\leq \lfloor u_n\rfloor\frac{\lambda(I_n)}{u_n}\leq\sigma_n\leq\lceil u_n\rceil\frac{\lambda(I_n)}{u_n}\leq n^{-a}+\dfrac{n^{-a}}{u_n}.\label{Extra}
  \end{equation}
    
    Since $u_n^{-1}$ is summable, it holds $\displaystyle\sum_{n=1}^N\sigma_n\sim \sum_{n=1}^Nn^{-a}\sim N^{1-a}$. 

Notice that for any $n<m$ and $i\in \{0,1,\hdots,\lfloor u_n\rfloor-1\}$, the length of the interval $I_{i,n}$ is $n^{-a}u_n^{-1}$, so $I_{i,n}$ fully contains  
    $$ \Big\lfloor \dfrac{n^{-a}}{u_n}/\dfrac{1}{u_m}\Big\rfloor=\Big\lfloor \dfrac{u_m}{n^au_n}\Big\rfloor$$ consecutive intervals of length $1/u_m$, and it might also contains an extra final interval of smaller length. For each interval $J$ of length $u_m^{-1}$ fully contained in $I_{i,n}$,  $$\lambda(I_{i,n} \cap J\cap  u_m^{-1}(I_{m}))=\frac{1}{m^au_m}.$$ Therefore, we can estimate
    \begin{equation}
        \Big\lfloor \dfrac{u_m}{n^au_n}\Big\rfloor\dfrac{1}{m^au_m}\leq \lambda\Big(I_{i,n}\cap u_m^{-1}(I_m)\Big)\leq\Big\lceil \dfrac{u_m}{n^au_n}\Big\rceil\dfrac{1}{m^au_m}.\label{Lemma_estimation_1_eq2}
    \end{equation}
    
    Also, note that $I_{\lfloor u_n \rfloor,n}$ has smaller measure than the others intervals $I_{i,n}$, so the upper bound in (\ref{Lemma_estimation_1_eq2}) also works for $i=\lfloor u_n \rfloor$, and in that case, we use $0$ as lower bound.  Using (\ref{Lemma_estimation_1_eq2}) in (\ref{Lemma_estimation_1_eq1}), we obtain that

   $$\Big\lfloor \dfrac{u_m}{n^au_n}\Big\rfloor \dfrac{1}{m^au_m}\lfloor u_n\rfloor\leq \lambda\Big(u_n^{-1}(I_{n})\cap u_m^{-1}(I_{m})\Big)\leq \Big\lceil \dfrac{u_m}{n^au_n}\Big\rceil\dfrac{1}{m^au_m}\Big(\lfloor u_n\rfloor+1\Big).$$ 
   We remove the integer parts to write
        $$\Big( \dfrac{u_m}{n^au_n}-1\Big)\dfrac{1}{m^au_m}\Big(u_n-1\Big)\leq \lambda\Big(u_n^{-1}(I_{n})\cap u_m^{-1}(I_{m})\Big)\leq \Big( \dfrac{u_m}{n^au_n}+1 \Big)\dfrac{1}{m^au_m}\Big(u_n+1\Big).$$
        Since
        $$\Big( \dfrac{u_m}{n^au_n}\pm1 \Big)\dfrac{1}{m^au_m}\Big(u_n\pm1\Big)=\dfrac{1}{n^am^a}\pm \dfrac{u_n}{m^{a}u_m}\pm\dfrac{1}{n^am^au_n}+\dfrac{1}{m^au_m},$$
     by rearranging the terms, using (\ref{Extra}) and using that $u_n^{-1}$ and $n^{-a}$ are decreasing sequences, it is easy to see that
    \begin{align*}
        \Big|\lambda\Big(u_n^{-1}(I_{n})\cap u_m^{-1}(I_m)\Big)-\sigma_{n}\sigma_{m} \Big|&\leq \Big|\lambda\Big(u_n^{-1}(I_{n})\cap u_m^{-1}(I_m)\Big)-\dfrac{1}{n^{a}m^{a}} \Big|+\Big|\sigma_{n}\sigma_{m}-\dfrac{1}{n^am^a}\Big| \\
        &\ll m^{-a} \Big(\dfrac{u_n}{u_m}+\dfrac{1}{u_n} \Big).
    \end{align*}
\end{proof}
Considering the sublacunary property of $(u_n)_{n\in\N}$, the previous result shows that $\Cov(X_n,X_{n+d})$ decays rapidly. This decay enables us to establish a version of the Law of Large Numbers for $(X_n)_{n\in\N}$, which describes the asymptotic size of $\Lambda_y\cap [1,N]$. In particular, this partially extends LeVeque's theorem (\ref{Theorem_LeVeque}) in the lacunary case to the setting of real sequences.

\begin{proposition}
 Let $a\in (0,1)$ and $\delta\in [0,1)$ such that $a+\delta<1$. Let $(u_n)_{n\in\N}\subseteq \R$ be a $\delta$-sublacunary sequence, and let $(I_n)_{n\in\N}$ be a sequence of sets such that $\lambda(I_n)=n^{-a}$, and there is $\ell\in\N$ such that each $I_n\subseteq [0,1]$ is union of at most $\ell$ intervals.  
 
 Then, for almost every $y\in[0,1]$,
    $$\lim_{N\to\infty}\dfrac{|\{n\in [1,N]:\ \{u_ny\}\in I_n\}|}{N^{1-a}}=1.$$\label{Proposition_LLN}
\end{proposition}

 \begin{proof} 
      For $n\in \N$, we define $X_n(y):=\mathbbm{1}_{I_n}(\{u_ny\})$ and note that $$\dfrac{|\{n\in [1,N]:\ \{u_ny\}\in I_n\}|}{N^{1-a}}=\dfrac{1}{N^{1-a}}\sum_{n=1}^NX_n(y).$$
   Since Lemma \ref{Lemma_estimation_1} yields to $\sum_{n=1}^N\mathbb{E}(X_n)\sim N^{1-a}$, to prove that $$\lim_{N\to\infty}\dfrac{1}{N^{1-a}}\sum_{n=1}^NX_n(y)=1$$ for almost every $y\in [0,1]$, it suffices verify the assumptions of \cite[Lemma 2.3]{DMS}. In other words, we need to find some $\varepsilon>0$ such that $$\sum_{n=1}^N\sum_{m=1}^N|\Cov(X_n,X_{n+m})|\ll N^{2-2a-\varepsilon}.$$

      Let $N\in\N$ and $n,m\leq N$. Using \cref{Lemma_sublacunary} and \cref{Lemma_estimation_1}, we can find $\beta>1$ such that
      \[|\Cov(X_n,X_{n+m})|\ll (n+m)^{-a}\Big(\dfrac{1}{u_n}+\frac{u_n}{u_{n+m}}\Big)\ll \frac{m^{-a}}{u_n}+n^{-a}\beta^{-\frac{m}{(n+m)^\delta}}.\]
      Summing up, we obtain
      \begin{equation*}
          \begin{split}
\sum_{n=1}^N\sum_{m=1}^N|\Cov(X_n,X_{n+m})|&\ll \sum_{n=1}^Nu_n^{-1}\sum_{m=1}^Nm^{-a}+\sum_{n=1}^Nn^{-a}\sum_{m=1}^N \beta^{-\frac{m}{(2N)^\delta}}\\
&\ll N^{1-a}+N^{1-a}\dfrac{1}{1-\beta^{-1/(2N)^\delta}}\\
&\ll N^{1-a+\delta}=N^{2-2a-\varepsilon},
\end{split}
      \end{equation*}
where $\varepsilon:=1-a-\delta>0$, concluding the result.
 \end{proof}

To control dependence in the study of pointwise convergence of random ergodic averages, we aim to establish decay estimates for higher-order correlations of the random variables. In particular, we will require bounds on the four-fold correlation. Following a similar strategy to  \cref{Lemma_estimation_1}, we obtain the next inequality.

 \begin{lemma}  Let $a\in (0,1)$ and $\delta\in [0,1)$. Let $(u_n)_{n\in\N}\subseteq \R$ be a $\delta$-sublacunary sequence, and let $(I_n)_{n\in\N}$ be a sequence of sets such that $\lambda(I_n)=n^{-a}$, and there is $\ell\in\N$ such that each $I_n\subseteq [0,1]$ is union of at most $\ell$ intervals.  
 
    Then, considering the notation of (\ref{Random_variables}), for any $n_1< n_2< n_3< n_4$,
\begin{align*}
    \left|\mathbb{E}(Y_{n_1}Y_{n_2}Y_{n_3}Y_{n_4})\right|
\ll n_3^{-a}n_2^{-a}\Big(
\frac{1}{u_{n_1}}
+ \frac{u_{n_1}}{u_{n_2}}
+ \frac{u_{n_2}}{u_{n_3}}
\Big)+n_3^{-a}\frac{u_{n_3}}{u_{n_4}}\Big(n_1^{-a}+\frac{1}{u_{n_1}}+\frac{u_{n_2}}{u_{n_3}}\Big).
\end{align*}\label{Lemma_estimates_2}
 \end{lemma}

 \begin{proof} Similarly to \cref{Lemma_estimation_1}, for clarity in the exposition  we assume that $\ell=1$ and every interval is open, emphasizing that the proof adapts to the general case. For $n\in\N$, we will write $\gamma_n:=n^{-a}$ and we will constantly use that $(u_n^{-1})_{n\in\N}$, $(\sigma_n)_{n\in\N}$ and $(\gamma_n)_{n\in\N}$ are decreasing, and that $\sigma_n\ll \gamma_n$.
 
   For $n_1<n_2<n_3<n_4$, we can write
\begin{equation}
\begin{split}
\mathbb{E}\left( Y_{n_1} Y_{n_2} Y_{n_3} Y_{n_4} \right) &=\Cov(X_{n_1},X_{n_2}X_{n_3}X_{n_4})-\sigma_{n_4}\Cov(X_{n_1},X_{n_2}X_{n_3})\\
&\quad-\sigma_{n_3}\Cov(X_{n_1},X_{n_2}X_{n_4})+\sigma_{n_3}\sigma_{n_4}\Cov(X_{n_1},X_{n_2})\\
&\quad-\sigma_{n_2}\Cov(X_{n_1},X_{n_3}X_{n_4})+\sigma_{n_2}\sigma_{n_4}\Cov(X_{n_1},X_{n_3})\\
&\quad+\sigma_{n_2}\sigma_{n_3}\Cov(X_{n_1},X_{n_4}).
\end{split}  \label{Lemma_covariance_eq1} 
\end{equation}
 By denoting $\xi(n_1,n_2):=\gamma_{n_2}(u_{n_1}^{-1}+u_{n_1}u_{n_2}^{-1})$, we use \cref{Lemma_estimation_1}  to obtain 
\begin{equation}
\begin{split}
 &\Big| \sigma_{n_3}\sigma_{n_4}\Cov(X_{n_1},X_{n_2})+
\sigma_{n_2}\sigma_{n_4}\Cov(X_{n_1},X_{n_3})
+\sigma_{n_2}\sigma_{n_3}\Cov(X_{n_1},X_{n_4})\Big|\\
&\ll \sigma_{n_2}\sigma_{n_3}\xi(n_1,n_2)\ll \gamma_{n_2}\gamma_{n_3}\xi(n_1,n_2).\label{Lemma_covariance_eq2} 
\end{split}
\end{equation}

On the other hand, for any $n< m<r$, 
\begin{equation}
\begin{split}
    |\Cov(X_{n},X_mX_r)|\leq& \left|\mathbb{E}(X_nX_mX_r)-\sigma_n\sigma_m\sigma_r\right|+\sigma_n\left|\Cov(X_m,X_r)\right|\\
    \ll&  \left|\mathbb{E}(X_nX_mX_r)-\gamma_n\gamma_m\gamma_r\right|\\
    &+\left|\sigma_n\sigma_m\sigma_r-\gamma_n\gamma_m\gamma_r\right|+\sigma_n\xi(m,r).
\end{split}\label{Lemma_covariance_eq3} 
\end{equation}

Seeking to estimate $\mathbb{E}(X_nX_mX_r)=\lambda\Big(u_n^{-1}(I_{n})\cap u_m^{-1}(I_m)\cap u_r^{-1}(I_r)\Big),$
    we rely on the same argument (and notation) used in the proof of \cref{Lemma_estimation_1}. 
    For any $j\in\{0,\hdots, \lfloor u_m\rfloor-1\}$, 
    \begin{equation}
        \Big\lfloor \frac{\gamma_m u_r}{u_m}\Big\rfloor \frac{\gamma_r}{u_r}\leq \lambda(I_{j,m}\cap u_r^{-1}(I_r))\leq  \dfrac{\gamma_r}{u_r}\Big\lceil \frac{\gamma_m u_r}{u_m}\Big\rceil.\label{Lemma_covariance_eq4}
    \end{equation}
  Also, the interval $I_{\lfloor u_m \rfloor,m}$ is smaller than the others. In such a case, the upper bound still holds and we use $0$ as lower bound.

     For any $i\in\{0,\hdots, \lfloor u_n\rfloor-1\}$,  let $k\in\N$ be the least integer such that
     $\frac{k}{u_m}\in I_{i,n}$ and let $k'\in \N$ be the greatest integer such that $\frac{k'}{u_m}\in I_{i,n}$. 
     Notice that $I_{j,m}\subseteq I_{i,n}$ for every $j\in \{k,\cdots,k'-1\}$, and $I_{i,n}\cap I_{j,m}=\emptyset$ if $j<k-1$ or $j>k'$. Therefore,
\begin{equation}
    \sum_{j=k}^{k'-1}\lambda(I_{j,m}\cap u_r^{-1}(I_r))\leq \lambda(I_{i,n}\cap u_m^{-1}(I_m)\cap u_r^{-1}(I_r))\leq \sum_{j=k-1}^{k'}\lambda( I_{j,m}\cap u_r^{-1}(I_r)).\label{Extra1}
\end{equation}
Since $I_{i,n}$ has length $\gamma_n/u_n$, we have that
\begin{equation}
    \frac{\gamma_nu_m}{u_n}-2\leq k'-k\leq \frac{\gamma_nu_m}{u_n}.\label{Extra2}
\end{equation}

Using (\ref{Lemma_covariance_eq4}), (\ref{Extra1}) and (\ref{Extra2}), we conclude that for every $i\in \{0,\cdots,\lfloor u_n\rfloor-1\}$,
     $$\Big(\dfrac{u_r\gamma_m}{u_m}- 2\Big)\Big(\dfrac{u_m\gamma_n}{u_n}- 2\Big)\dfrac{\gamma_r}{u_r}\leq \lambda(I_{i,n}\cap u_m^{-1}(I_m)\cap u_r^{-1}(I_r))\leq \Big(\dfrac{u_r\gamma_m}{u_m}+ 2\Big)\Big(\dfrac{u_m\gamma_n}{u_n}+ 2\Big)\dfrac{\gamma_r}{u_r}.$$
It follows that 
    
    \begin{equation}
        \begin{split}
            \Big(\dfrac{u_r\gamma_m}{u_m}- 2\Big)\Big(\dfrac{u_m\gamma_n}{u_n}- 2\Big)\dfrac{\gamma_r}{u_r}\Big(u_n- 1\Big)&\leq \lambda\Big(u_n^{-1}(I_{n})\cap u_m^{-1}(I_m)\cap u_r^{-1}(I_r)\Big)\\
            &\leq \Big(\dfrac{u_r\gamma_m}{u_m}+ 2\Big)\Big(\dfrac{u_m\gamma_n}{u_n}+ 2\Big)\dfrac{\gamma_r}{u_r}\Big(u_n+ 1\Big).
        \end{split}\label{Lemma_covariance_eq5}
    \end{equation}
    Expanding the last expression,
\begin{equation}
\begin{split}
&\Big(\dfrac{u_r\gamma_m}{u_m}\pm 2\Big)\Big(\dfrac{u_m\gamma_n}{u_n}\pm 2\Big)\dfrac{\gamma_r}{u_r}\Big(u_n\pm 1\Big)\\
&=\gamma_n\gamma_m\gamma_r\pm \dfrac{\gamma_n\gamma_m\gamma_r}{u_n}\pm 2\dfrac{\gamma_m\gamma_ru_n}{u_m}+2\dfrac{\gamma_m\gamma_r}{u_m}\\
&\pm2\dfrac{u_m\gamma_n\gamma_r}{u_r}+2\dfrac{u_m\gamma_n\gamma_r}{u_nu_r}+4\dfrac{u_n\gamma_r}{u_r}\pm 4\dfrac{\gamma_r}{u_r}.
\end{split}\label{expansion}
\end{equation}

By monotonicity of the sequences, we deduce that
\begin{equation}| \mathbb{E}(X_nX_mX_r)-\gamma_n\gamma_m\gamma_r|\ll \xi(n,m,r):=\frac{\gamma_r}{u_n}
+ \frac{\gamma_m \gamma_r u_n}{u_m}
+ \frac{\gamma_n \gamma_r u_m}{u_r}
+ \frac{\gamma_r u_n}{u_r}.\label{Extra4}
\end{equation}
Also, from (\ref{Extra}), we can  see that \begin{equation*}
 \left|\sigma_n\sigma_m\sigma_r-\gamma_n\gamma_m\gamma_r\right|\ll \dfrac{\gamma_n\gamma_m\gamma_r}{u_n}\ll \xi(n,m,r).\label{Extra0}   
\end{equation*}
Hence, we conclude from (\ref{Lemma_covariance_eq3}) that
\begin{align*}
    &\Big|-\sigma_{n_4}\Cov(X_{n_1},X_{n_2}X_{n_3})
-\sigma_{n_3}\Cov(X_{n_1},X_{n_2}X_{n_4})
-\sigma_{n_2}\Cov(X_{n_1},X_{n_3}X_{n_4})\Big|\\
&\ll\sigma_{n_4}\xi(n_1,n_2,n_3)+\sigma_{n_3}\xi(n_1,n_2,n_4)+\sigma_{n_2}\xi(n_1,n_3,n_4)\\
&\qquad +\sigma_{n_4}\sigma_{n_1}\xi(n_2,n_3)+\sigma_{n_3}\sigma_{n_1}\xi(n_2,n_4)+\sigma_{n_2}\sigma_{n_1}\xi(n_3,n_4).
\end{align*}
It is not hard to see that 
\begin{equation*}
    \begin{split}
    \sigma_{n_4}\sigma_{n_1}\xi(n_2,n_3)+\sigma_{n_3}\sigma_{n_1}\xi(n_2,n_4)+\sigma_{n_2}\sigma_{n_1}\xi(n_3,n_4)\ll \gamma_{n_1}\gamma_{n_2}\gamma_{n_3}\Big(\frac{1}{u_{n_2}}+\frac{u_{n_2}}{u_{n_3}}+\frac{u_{n_3}}{u_{n_4}}\Big).    
    \end{split}
\end{equation*}
On the other hand, 
\begin{equation*}
    \begin{split}
&\sigma_{n_4}\xi(n_1,n_2,n_3)+\sigma_{n_3}\xi(n_1,n_2,n_4)+\sigma_{n_2}\xi(n_1,n_3,n_4)\\
&\ll\gamma_{n_2}\gamma_{n_3}\Big(
\frac{1}{u_{n_1}}
+ \frac{\gamma_{n_4}\,u_{n_1}}{u_{n_2}}
+ \frac{\gamma_{n_1}\,u_{n_2}}{u_{n_3}}
+ \frac{\gamma_{n_1}\,u_{n_3}}{u_{n_4}}
+ \frac{u_{n_1}}{u_{n_3}}
\Big).
    \end{split}
\end{equation*}
Putting the last three inequalities together and using the monotonicity of the sequences, it follows that 
\begin{equation}
    \begin{split}
        &\Big|-\sigma_{n_4}\Cov(X_{n_1},X_{n_2}X_{n_3})
-\sigma_{n_3}\Cov(X_{n_1},X_{n_2}X_{n_4})
-\sigma_{n_2}\Cov(X_{n_1},X_{n_3}X_{n_4})\Big|\\
&\ll \gamma_{n_2}\gamma_{n_3}\Big(
\frac{1}{u_{n_1}}
+ \frac{\gamma_{n_4}\,u_{n_1}}{u_{n_2}}
+ \frac{\gamma_{n_1}\,u_{n_2}}{u_{n_3}}
+ \frac{\gamma_{n_1}\,u_{n_3}}{u_{n_4}}
+ \frac{u_{n_1}}{u_{n_3}}
\Big).
    \end{split}\label{Lemma_covariance_eq6}
\end{equation}

Finally, we proceed to estimate
$$\Cov(X_{n_1},X_{n_2}X_{n_3}X_{n_4})=\mathbb{E}(X_{n_1}X_{n_2}X_{n_3}X_{n_4})-\sigma_{n_1}\mathbb{E}(X_{n_2}X_{n_3}X_{n_4}).$$ 
Since
\begin{align*}
|\Cov(X_{n_1},X_{n_2}X_{n_3}X_{n_4})|\ll& \Big|\mathbb{E}(X_{n_1}X_{n_2}X_{n_3}X_{n_4})-\gamma_{n_1}\gamma_{n_2}\gamma_{n_3}\gamma_{n_4}\Big|+\sigma_{n_1}\xi(n_2,n_3,n_4)\\
&+\Big|\gamma_{n_1}\gamma_{n_2}\gamma_{n_3}\gamma_{n_4}-\sigma_{n_1}\sigma_{n_2}\sigma_{n_3}\sigma_{n_4}\Big|\\
&\ll \Big|\mathbb{E}(X_{n_1}X_{n_2}X_{n_3}X_{n_4})-\gamma_{n_1}\gamma_{n_2}\gamma_{n_3}\gamma_{n_4}\Big|+\sigma_{n_1}\xi(n_2,n_3,n_4),
\end{align*}
it only remains to study $\mathbb{E}(X_{n_1}X_{n_2}X_{n_3}X_{n_4})$. To do this, we follow a similar idea as in (\ref{Lemma_covariance_eq5}), obtaining
\begin{equation}
    \begin{split}
    &\Big(\dfrac{\gamma_{n_3}u_{n_4}}{u_{n_3}}- 2\Big)\Big(\dfrac{\gamma_{n_2}u_{n_3}}{u_{n_2}}- 2\Big)\Big(\dfrac{\gamma_{n_1}u_{n_2}}{u_{n_1}}- 2\Big)\dfrac{\gamma_{n_4}}{u_{n_4}}\Big(u_{n_1}- 1\Big)\\
        &\leq \mathbb{E}(X_{n_1}X_{n_2}X_{n_3}X_{n_4})\\
        &\leq \Big(\dfrac{\gamma_{n_3}u_{n_4}}{u_{n_3}}+ 2\Big)\Big(\dfrac{\gamma_{n_2}u_{n_3}}{u_{n_2}}+ 2\Big)\Big(\dfrac{\gamma_{n_1}u_{n_2}}{u_{n_1}}+ 2\Big)\dfrac{\gamma_{n_4}}{u_{n_4}}\Big(u_{n_1}+ 1\Big).
    \end{split}\label{Lemma_covariance_eq7}
\end{equation}
Conveniently, we can rewrite
\begin{equation*}
    \begin{split}
        &\Big(\dfrac{\gamma_{n_3}u_{n_4}}{u_{n_3}}\pm 2\Big)\Big(\dfrac{\gamma_{n_2}u_{n_3}}{u_{n_2}}\pm 2\Big)\Big(\dfrac{\gamma_{n_1}u_{n_2}}{u_{n_1}}\pm 2\Big)\dfrac{\gamma_{n_4}}{u_{n_4}}\Big(u_{n_1}\pm 1\Big)\\
        &=\Big(\gamma_{n_4}\pm2\dfrac{\gamma_{n_4}}{\gamma_{n_3}}\dfrac{u_{n_3}}{u_{n_4}}\Big) \Bigg(\Big(\dfrac{\gamma_{n_2}u_{n_3}}{u_{n_2}}\pm 2\Big)\Big(\dfrac{\gamma_{n_1}u_{n_2}}{u_{n_1}}\pm 2\Big)\dfrac{\gamma_{n_3}}{u_{n_3}}\Big(u_{n_1}\pm 1\Big)\Bigg),
    \end{split}
\end{equation*}
This makes explicit the connection with the expression obtained in (\ref{expansion}). Then, since $\frac{\gamma_{n_4}}{\gamma_{n_3}}\in (0,1)$ and recalling (\ref{Extra4}), we can then derive that
\begin{equation}
    \begin{split}
        &|\mathbb{E}(X_{n_1}X_{n_2}X_{n_3}X_{n_4})-\gamma_{n_1}\gamma_{n_2}\gamma_{n_3}\gamma_{n_4}|\\
        &\ll \gamma_{n_1}\gamma_{n_2}\gamma_{n_3}\frac{u_{n_3}}{u_{n_4}}+\xi(n_1,n_2,n_3)\Big(\gamma_{n_4}+\frac{u_{n_3}}{u_{n_4}}\Big)\\
        &\ll \gamma_{n_4}\xi(n_1,n_2,n_3) +\gamma_{n_3}\frac{u_{n_3}}{u_{n_4}}\Big(\gamma_{n_1}\gamma_{n_2}+\frac{1}{u_{n_1}}+\frac{\gamma_{n_2}u_{n_1}}{u_{n_2}}+\frac{u_{n_2}}{u_{n_3}}\Big).
        \end{split}\label{Lemma_covariance_eq8}
\end{equation}
Combining (\ref{Lemma_covariance_eq1}) with  (\ref{Lemma_covariance_eq2}), (\ref{Lemma_covariance_eq6}) and (\ref{Lemma_covariance_eq8}), it follows
\begin{align*}
    \left|\mathbb{E}(Y_{n_1}Y_{n_2}Y_{n_3}Y_{n_4})\right|
&\ll \gamma_{n_3}\Bigg(\gamma_{n_2}\Big(
\frac{1}{u_{n_1}}
+ \frac{\gamma_{n_4}\,u_{n_1}}{u_{n_2}}
+ \frac{\gamma_{n_1}\,u_{n_2}}{u_{n_3}}
+ \frac{\gamma_{n_1}\,u_{n_3}}{u_{n_4}}
+ \frac{u_{n_1}}{u_{n_3}}
\Big)\\
+\gamma_{n_2}^2&\Big(\frac{1}{u_{n_1}}+\frac{u_{n_1}}{u_{n_2}}\Big)+\frac{u_{n_3}}{u_{n_4}}\Big(\gamma_{n_1}\gamma_{n_2}+\frac{1}{u_{n_1}}+\frac{\gamma_{n_2}u_{n_1}}{u_{n_2}}+\frac{u_{n_2}}{u_{n_3}}\Big)\Bigg).\\
\end{align*}
The result follows by rearranging the terms, using that $\frac{u_{n_1}}{u_{n_3}}\leq \frac{u_{n_1}}{u_{n_2}}$ and $\gamma_n\in (0,1)$.
 \end{proof}

\section{Proof of \texorpdfstring{\cref{thmA}}{Theorem A}}
In this section, we present the proof of \cref{thmA}. 

To study $\overline{\dimM}(A\cap \Lambda_y)$, we need to analyze the behavior of
\begin{equation}
    \dfrac{|A\cap \Lambda_y\cap [1,N]|}{N^\gamma}=\dfrac{1}{N^\gamma}\sum_{n=1}^NX_{n}(y)\mathbbm{1}_{A}(n).\label{A_1}
\end{equation}
Our strategy is to compare (\ref{A_1}) with its expectation. As a first step, we need to understand the behavior of the expectation itself, 
which is provided by the following result.

  \begin{proposition}
 Let $A\subseteq \mathbb{N}$ and $a\in (0,1)$. Then,
    \begin{equation}\limsup_{N\to \infty}\dfrac{1}{N^{\gamma}}\sum_{n=1}^Nn^{-a}\mathbbm{1}_{A}(n)=\begin{cases}
 0 &\text{ if } \gamma>\overline{\gamma} \\ 
 +\infty &\text{ if } 0<\gamma<\overline{\gamma},\\
\end{cases}\label{thmA_prop_eq}
\end{equation} 
where $\displaystyle\overline{\gamma}=\max\left\{0, \overline{\dimM}(A)-a \right\}.$ 
     
     If $\dimM(A)$ is well-defined, the result holds for $\lim$ instead of $\limsup$, i.e.,
     \[\lim_{N\to \infty}\dfrac{1}{N^{\gamma}}\sum_{n=1}^Nn^{-a}\mathbbm{1}_{A}(n)=\limsup_{N\to \infty}\dfrac{1}{N^{\gamma}}\sum_{n=1}^Nn^{-a}\mathbbm{1}_{A}(n).\]\label{thmA_prop}
\end{proposition}

\begin{proof}
 Let $N\in\N$ and $\gamma>0$. Using summation by parts, we can write
    \begin{equation*}
    \begin{split}
        \dfrac{1}{N^\gamma}\sum_{n=1}^N n^{-a}\mathbbm{1}_{A}(n)&=S_1(\gamma,N)+S_2(\gamma,N),
    \end{split}
    \end{equation*}
where
\begin{align*}
    &S_1(\gamma,N):=\dfrac{1}{N^{\gamma+a}}\sum_{n=1}^N\mathbbm{1}_{A}(n) \quad \text{and} \\
    &S_2(\gamma,N):=\dfrac{1}{N^\gamma}\sum_{M=1}^{N-1}(M^{-a}-(M+1)^{-a})\sum_{n=1}^M\mathbbm{1}_{A}(n).
\end{align*}
 If $\gamma>\overline{\gamma},$ then $\gamma+a>\overline{\dimM}(A)$. Hence,

\[\limsup_{N\to \infty}\dfrac{1}{N^{\gamma+a}}\sum_{n=1}^N\mathbbm{1}_{A}(n)=\limsup_{N\to \infty}\dfrac{|A\cap [1,N]|}{N^{\gamma+a}}=0.\]
On the other hand, let $\gamma\in (0,\overline{\gamma})$. In such a case, $\overline{\gamma}=\overline{\dimM}(A)-a>0$, and therefore, $\gamma+a<\overline{\dimM}(A)$.
We conclude that

\begin{equation}
    \limsup_{N\to \infty}S_1(\gamma,N)=\begin{cases}
0 & \text{ if } \gamma>\overline{\gamma} \\ 
+\infty & \text{ if } 0<\gamma<\overline{\gamma}.
\end{cases}\label{Eq_S1}
\end{equation}

Since $n^{-a}$ is decreasing, $S_2(\gamma,N)\geq 0$. In particular, for any $\gamma\in (0,\overline{\gamma}),$
\begin{equation}
    \limsup_{N\to \infty}\dfrac{1}{N^\gamma}\sum_{n=1}^N n^{-a}\mathbbm{1}_{A}(n)=\limsup_{N\to \infty} S_1(\gamma,N)+S_2(\gamma,N)=+\infty.\label{Eq_case_1}\end{equation}

Now, let $\gamma>\overline{\gamma}$ and define  $\eta:=\gamma-\overline{\gamma}>0$. By the mean value theorem,
\begin{align*}
  S_2(\gamma,N)
  &\ll  \dfrac{1}{N^\gamma}\sum_{M=1}^{N-1}M^{-1-a}\sum_{n=1}^M\mathbbm{1}_{A}(n).\end{align*}
  Noticing that the sequence  $$\dfrac{1}{M^{\overline{\dimM}(A)+\eta/2}}\sum_{n=1}^M\mathbbm{1}_{A}(n)$$
  is bounded, we have that
\begin{align*}
  \dfrac{1}{N^\gamma}\sum_{M=1}^{N-1}M^{-1-a}\sum_{n=1}^M\mathbbm{1}_{A}(n)&=\dfrac{1}{N^\gamma}\sum_{M=1}^{N-1}M^{-1-a+\overline{\dim}_\text{M}(A)+\eta/2}\dfrac{1}{M^{\overline{\dimM}(A)+\eta/2}}\sum_{n=1 }^M\mathbbm{1}_{A}(n)  \\
  &\ll \dfrac{1}{N^\gamma}\sum_{M=1}^{N-1}M^{-1-a+\overline{\dimM}(A)+\eta/2} \\
  &=\dfrac{1}{N^{\overline{\gamma}+\eta}}\sum_{M=1}^{N-1}M^{-1-a+\overline{\dimM}(A)+\eta/2}.
\end{align*}

Using $N^{-\overline{\gamma}-2\eta/3}\leq M^{-\overline{\gamma}-2\eta/3} $, it follows 
\[S_2(\gamma,N)\ll \dfrac{1}{N^{\overline{\gamma}+\eta}}\sum_{M=1}^{N-1}M^{-1-a+\overline{\dimM}(A)+\eta/2}\leq \dfrac{1}{N^{\eta/3}}\sum_{M=1}^{N-1} M^{-1-a+\overline{\dimM}(A)+\eta/2}M^{-\overline{\gamma}-2\eta/3}.\]
Since $-1-\eta/6-(\overline{\gamma}-(\overline{\dimM}(A)-a))<-1$, it follows that 
\begin{equation}
    \lim_{N\to\infty}S_2(\gamma,N)=0\label{Eq_S2}
\end{equation}
for every $\gamma>\overline{\gamma}.$
Putting together (\ref{Eq_S1}) and (\ref{Eq_S2}), we conclude that for every $\gamma>\overline{\gamma}$,
\begin{equation}
    \lim_{N\to \infty}\dfrac{1}{N^{\gamma}}\sum_{n=1}^Nn^{-a}\mathbbm{1}_{A}(n)\leq\limsup_{N\to \infty} S_1(\gamma,N)+S_2(\gamma,N)=0.\label{Eq_case_2}
\end{equation}
From (\ref{Eq_case_1}) and (\ref{Eq_case_2}), we obtain (\ref{thmA_prop_eq}). If $\dimM(A)$ is well-defined, we can immediately replace the $\limsup$ by $\lim$ in (\ref{Eq_S1}), and then, we can do it in (\ref{Eq_case_1}). This leads to the conclusion.
\end{proof}

We are now in a position to prove Theorem \ref{thmA}.
\begin{proof}[Proof of \cref{thmA}]   Notice that the first part, (\ref{thmA_first}), follows from Proposition \ref{Proposition_LLN}.

Let $A\subseteq \N$ such that $$\overline{\gamma}:=\max\{0,\overline{\dimM}(A)-a\}\geq \delta.$$ Using the notation of (\ref{Random_variables}), for $N\in \N$ and $\gamma>0$, we define
    $$V_{\gamma,N}(y):=\dfrac{1}{N^\gamma}\sum_{n=1}^N\mathbbm{1}_{A}(n)\cdot Y_n(y).$$  Hence, using that $\sigma_n\ll n^{-a},$
    \begin{equation}
 \mathbb{E}(V_{\gamma,N}^2)\ll \dfrac{1}{N^{2\gamma}}\left(\sum_{n=1}^Nn^{-a}\mathbbm{1}_{A}(n)+\sum_{n=1 }^N\sum_{m=n+1}^N \mathbbm{1}_{A}(n)\mathbbm{1}_{A}(m)\mathbb{E}(Y_nY_m)\right). \label{TheoremA_eq1}
\end{equation}

If $\gamma>\overline{\gamma}/2$, Proposition \ref{thmA_prop} leads to 
 
\begin{equation}\dfrac{1}{N^{2\gamma}}\sum_{n=1}^N n^{-a}\mathbbm{1}_A(n)=\dfrac{1}{N^{\gamma-\overline{\gamma}/2}}\left (\dfrac{1}{N^{\overline{\gamma}+(\gamma-\overline{\gamma}/2)}}\sum_{n=1}^N n^{-a}\mathbbm{1}_A(n)\right)\ll N^{-(\gamma-\overline{\gamma}/2)}.\label{TheoremA_eq2}
\end{equation}
On the other hand, using  \cref{Lemma_sublacunary} and \cref{Lemma_estimation_1}, we have that

 \begin{equation*}
        \sum_{n=1 }^N\sum_{m=n+1}^N \mathbbm{1}_{A}(n)\mathbbm{1}_{A}(m)\mathbb{E}(Y_nY_{m})\ll\sum_{n=1 }^N\sum_{d=1}^N \mathbbm{1}_{A}(n)\mathbbm{1}_{A}(n+d) (n+d)^{-a}\Big(u_{n}^{-1}+\beta^{-\frac{d}{(2N)^\delta}}\Big),
\end{equation*}
for some $\beta>1$. Noticing that 
$$\sum_{n=1 }^N\sum_{d=1}^N \mathbbm{1}_{A}(n)\mathbbm{1}_{A}(n+d) (n+d)^{-a}u_{n}^{-1}\ll \sum_{k=1}^{2N}k^{-a}\mathbbm{1}_{A}(k)$$
and
$$\sum_{n=1 }^N\sum_{d=1}^N \mathbbm{1}_{A}(n)\mathbbm{1}_{A}(n+d) (n+d)^{-a}\beta^{-\frac{d}{(2N)^\delta}}\ll N^\delta \sum_{n=1}^Nn^{-a}\mathbbm{1}_A(n),$$
we use Proposition \ref{thmA_prop} to obtain that for every $\varepsilon>0$,
\begin{equation}
   \dfrac{1}{N^{2\gamma}}\sum_{n=1 }^N\sum_{m=n+1}^N \mathbbm{1}_{A}(n)\mathbbm{1}_{A}(m)\mathbb{E}(Y_nY_{m})\ll \dfrac{N^\delta}{N^{2\gamma}}\sum_{n=1}^Nn^{-a}\mathbbm{1}_A(n)\ll N^{\delta+(\overline{\gamma}+\varepsilon)-2\gamma}.\label{TheoremA_eq3} 
\end{equation}
 Combining (\ref{TheoremA_eq1}), (\ref{TheoremA_eq2}) and (\ref{TheoremA_eq3}), it follows that for every $\gamma>\overline{\gamma}/2$ and $\varepsilon>0$,
\begin{equation}
    \mathbb{E}(V_{\gamma,N}^2)\ll  N^{-(\gamma-\overline{\gamma}/2)} +N^{\delta+(\overline{\gamma}+\varepsilon)-2\gamma}.\label{Equation_thmA_4}
\end{equation}
 We conclude that, for every $A\subseteq \N$, $\gamma>(\delta+\overline{\gamma})/2$ and $\eta>1$,
 $$\sum_{N=1}^\infty\mathbb{E}( V_{\gamma,\lfloor\eta^N\rfloor}^2)<\infty.$$
 By using the Markov Inequality and Borel-Cantelli,
\begin{equation}
\lim_{N\to\infty} V_{\gamma,\lfloor\eta^N\rfloor}(y)=0 \text{ for almost every $y\in [0,1]$}.
\end{equation}

Putting together this fact with Proposition \ref{thmA_prop}, we get that for every $\gamma>\overline{\gamma}\geq (\delta+\overline{\gamma})/2$ and $\eta>1$,
$$\lim_{N\to \infty}\dfrac{1}{\lfloor\eta^N\rfloor^{\gamma}}\sum_{n=1}^{\lfloor\eta^N\rfloor}X_n(y)\mathbbm{1}_{A}(n)=0\text{ for almost every $y\in [0,1]$}.$$
 Applying \cref{Lemma_lacunary_trick}, we conclude that for any $\gamma>\overline{\gamma}$, there is $\Omega_{\gamma,A}\subseteq [0,1]$ of full measure such that for all $y\in \Omega_{\gamma,A}$,
\begin{equation}
    \lim_{N\to \infty}\dfrac{1}{N^{\gamma}}\sum_{n=1}^{N}X_n(y)\mathbbm{1}_{A}(n)=\lim_{N\to\infty}\dfrac{|A\cap \Lambda_y\cap [1,N]|}{N^\gamma}=0.\label{thmA_key1}
\end{equation}

For the case $\frac{\delta+\overline{\gamma}}{2}<\gamma<\overline{\gamma}$, Proposition \ref{thmA_prop} allows us to find a subsequence $(k_n)_{n\in\N}$ such that 
$$ \lim_{N\to \infty}\dfrac{1}{(k_N)^{\gamma}}\sum_{n=1}^{k_N}\sigma_n\mathbbm{1}_{A}(n)=\infty.$$
Without loss of generality, we can suppose that $k_n\geq 2^n$. From (\ref{Equation_thmA_4}), $$\displaystyle\sum_{n=1}^\infty\mathbb{E}(V_{\gamma,k_n}^2)<+\infty,$$ thus
 $\displaystyle \lim_{N\to \infty} V_{\gamma,k_N}(y)=0$ almost everywhere. Therefore, 
 \[\lim_{N\to \infty}\dfrac{1}{(k_N)^\gamma}\sum_{n=1}^{k_N} \mathbbm{1}_A(n)X_n(y)=\infty\quad \text{for almost every $y\in[0,1]$}.\] In particular, for every $\frac{\delta+\overline{\gamma}}{2}<\gamma<\overline{\gamma}$, there is a set of full measure $\Omega_{\gamma,A}\subset[0,1]$ such that for all $y\in \Omega_{\gamma,A}$,

\begin{equation}
  \limsup_{N\to \infty}\dfrac{1}{N^\gamma}\sum_{n=1}^N \mathbbm{1}_A(n)X_n(y)=\infty.\label{thmA_key_2}
 \end{equation}
 By taking $$\Omega_A := 
 \bigcap_{\gamma\in (\frac{\delta+\overline{\gamma}}{2},1)\cap \mathbb{Q}}\Omega_{\gamma,A},$$ we conclude that for every $y\in \Omega_A$, \begin{equation*}\limsup_{N\to\infty}\dfrac{1}{N^\gamma}\sum_{n=1}^N \mathbbm{1}_A(n)X_n(y)=\limsup_{N\to \infty}\dfrac{1}{N^\gamma}|A\cap \Lambda_y\cap[1,N]|=\begin{cases}
 +\infty &\text{ if } \gamma \in (\frac{\delta+\overline{\gamma}}{2},\overline{\gamma})\\ 
 0 &\text{ if } \gamma\in (\overline{\gamma},1)
\end{cases}\end{equation*}
When $a+\delta<\overline{\dimM}(A)$, it holds $\overline{\gamma}=\overline{\dimM}(A)-a$ and $\overline{\gamma}>(\delta+\overline{\gamma})/2$. Then, $$\overline{\dimM}(A\cap \Lambda_y)=\overline{\dimM}(A)-a$$ for every $y\in \Omega_A$. Notice that if $\dimM(A)$ is well-defined, (\ref{thmA_key_2}) holds when replacing $\limsup$ by $\lim$, which allows us to conclude that $\dimM(A\cap \Lambda_y)$ is well-defined.

On the other hand, if $a>\overline{\dimM}(A)$ and $\delta=0$, we conclude that, for every $y\in \Omega_A$ and $\gamma>0=\overline{\gamma}$,$$\limsup_{N\to \infty}\dfrac{1}{N^\gamma}|A\cap \Lambda_y\cap[1,N]|=0,$$
obtaining that $\overline{\dimM}(A\cap \Lambda_y)=\dimM(A\cap\Lambda_y)=0.$

\end{proof}

\section{Proof of Theorem B}
Convergence of random ergodic averages in the lack of independence was firstly studied by Donoso, Maass, and the author in \cite{DMS}. While the random sequences under consideration in this work are of a different nature, the proof of \cref{thmB} is guided by the approach taken in \cite[Theorem~3.1]{DMS} (see also \cite[Theorem 1.1]{Frantzikinakis_Lesigne_Wierdl}). The essential distinction lies in the way in which dependence among the random variables is controlled, which in our setting relies on the estimates established in Lemma \ref{Lemma_estimates_2}.

The following lemma, which is implicitly contained in the proof of \cite[Theorem~3.1]{DMS}, gathers some steps we will reuse in our argument. For completeness, we outline the main points without reproducing the full details.

\begin{lemma} Let $(\Omega,\mathcal{F},\mathbb{P})$ be a probability space and $(X_n)_{n\in\N}$ be a sequence of random variables to values in $\{0,1\}$ such that $\mathbb{P}(X_n=1)\sim n^{-a}$ for some $a\in (0,1/2)$, and it holds
\begin{equation}
    \lim_{N\to\infty}\dfrac{1}{N^{1-a}}\sum_{n=1}^N\mathbb{E}(X_n)=\lim_{N\to \infty}\dfrac{1}{N^{1-a}}\sum_{n=1}^N X_n(\omega)=1 \text{ almost everywhere}.\label{Lemma_previous_paper_H1}
    \end{equation}

For $Y_n:=X_n-\mathbb{E}(X_n)$, if there exist $b\in (2a,1)$ and $\varepsilon>0$  such that
\begin{equation}
    \begin{split}
    \mathbb{E}\Bigg(\sum_{m=1}^{\lfloor N^{b}\rfloor}\Big|\sum_{n=1}^{N-m}Y_{n+m}Y_n\Big|\Bigg)\ll N^{1-2a+b-\varepsilon},\label{Lemma_previous_paper_H2}
    \end{split}
\end{equation}
then the random sequence $$\Big\{n\in\N:\ X_n(\omega)=1\Big\}=\Big\{a_1(\omega)<a_2(\omega)<\cdots\Big\}$$ is pointwise universally $L^2$-good and ergodic for almost every $\omega\in \Omega.$
    \label{Lemma_previous_paper}
\end{lemma}
\begin{proof} Let $(X,\mathcal{X},\mu,T)$ be a measure-preserving system and $f\in L^2(\mu)$.  
Notice that, to show 
\[\lim_{N\to\infty}\dfrac{1}{N}\sum_{n=1}^N f(T^{a_n(\omega)}x)=\mathbb{E}(f|\mathcal{I}(T))(x),\]
it is enough to show that 
\[\lim_{N\to\infty}\dfrac{1}{\sum_{n=1}^NX_n(\omega)}\sum_{n=1}^N X_n(\omega)f(T^{n}x)=\lim_{N\to\infty}\dfrac{1}{N^{1-a}}\sum_{n=1}^N X_n(\omega)f(T^{n}x)=\mathbb{E}(f|\mathcal{I}(T))(x).\]
On the other hand, \cite[Lemma A.3]{Frantzikinakis_Lesigne_Wierdl} and the ergodic theorem allows us to conclude that for almost every $x\in X$,
  $$\lim_{N\to\infty}\dfrac{1}{N^{1-a}}\sum_{n=1}^N\mathbb{E}(X_n)f(T^nx)=\lim_{N\to \infty}\dfrac{1}{N}\sum_{n=1}^Nf(T^nx)=\mathbb{E}(f|\mathcal{I}(T))(x).$$

  Hence, it is enough to show that for almost every $\omega\in \Omega$, for every measure-preserving system $(X,\mathcal{X},\mu,T)$ and every $f\in L^2(\mu)$, it holds that
   \begin{equation}
     \lim_{N\to\infty}\dfrac{1}{N^{1-a}}\sum_{n=1}^N(X_n(\omega)-\mathbb{E}(X_n))f(T^nx)=0\quad\text{
   for almost every $x\in X$ }\label{Lemma_previous_paper_eq1}
   \end{equation}

For $N\in \N$ and $\omega\in \Omega$ , we define
 $$V_{N}(\omega):=\left\|N^{-1+a}\sum_{n=1}^N Y_n(y)\cdot T^{n}f \right\|^2_{L^2(\mu)}.$$
 
    By using the Van der Corput lemma (see \cite[Lemma 2.3]{DMS}) with $M=\lfloor N^b\rfloor$, it follows that
        \begin{equation*}
          \begin{split}
          V_N(\omega)&\ll N^{-1+2a} M^{-1}\sum_{n=1}^{N}\| Y_n(\omega)\cdot T^{n}f\|_{L^2(\mu)}^2+V_{1,N}(\omega),
    \end{split}
    \end{equation*}
where \[V_{1,N}(\omega):=N^{-1+2a}M^{-1}\sum_{m=1}^{M}\left|\sum_{n=1}^{N-m}\int_{X} Y_{n+m}(\omega)\cdot Y_n(\omega)\cdot T^{n+m}f\cdot T^n f\text{d}\mu\right|.\]
Recalling that $f\in L^2(\mu)$, we can estimate $V_{1,N}$ by composing with $T^{-n}$ and applying the Cauchy--Schwarz inequality, obtaining that
\begin{equation*}
    V_N(\omega)\ll N^{-1+2a-b}\sum_{n=1}^NY_n^2(\omega)+V'_N(\omega),
\end{equation*}
where $$V'_N(\omega):=N^{-1+2a-b}\sum_{m=1}^M\Big|\sum_{n=1}^{N-m}Y_{n+m}(\omega)Y_n(\omega)\Big|.$$

It is worth emphasizing here that the right-hand side of the inequality no longer depends on the dynamical system under consideration. It is also possible to see from (\ref{Lemma_previous_paper_H1}) that $\sum_{n=1}^NY_n^2(\omega)\ll N^{1-a}$ almost surely. Then,
\begin{equation}
      V_N(\omega)\ll N^{a-b}+V'_N(\omega).\label{Lemma_previous_paper_eq2}
\end{equation}

Using the hypothesis (\ref{Lemma_previous_paper_H2}), we can choose $b\in (2a,1)$ and $\varepsilon>0$ such that 
$$N^{-1+2a-b}\mathbb{E}\Bigg(\sum_{m=1}^{\lfloor N^{b}\rfloor}\Big|\sum_{n=1}^{N-m}Y_{n+m}Y_n\Big|\Bigg)\ll N^{-\varepsilon}.$$
In particular, for any $\eta>1$, 
$$\sum_{N=1}^\infty\mathbb{E}(V'_{\lfloor\eta^N\rfloor})<\infty.$$
A classical application of Markov Inequality and Borel-Cantelli allows us to conclude that 
$$\sum_{N=1}^\infty V'_{\lfloor\eta^N\rfloor}(\omega)<\infty\quad \text{for almost every $\omega\in\Omega$}.$$

Therefore, from (\ref{Lemma_previous_paper_eq2}), there is a set of full measure $\Omega'$ (independent of the dynamical system) such that 
$$\sum_{N=1}^\infty V_{\lfloor(1+1/k)^N\rfloor}(\omega)<\infty$$
for every $\omega\in \Omega'$ and every $k\in\N$. It follows that for every $\omega\in \Omega'$, 
$$\lim_{N\to\infty}\dfrac{1}{\lfloor(1+1/k)^N\rfloor^{1-a}}\sum_{n=1}^{\lfloor(1+1/k)^N\rfloor}(X_n(\omega)-\mathbb{E}(X_n))f(T^nx)=0$$
for almost every $x\in X$ and every every $k\in\N$. We complete the proof by using \cref{Lemma_lacunary_trick}.
\end{proof}

We are now in position to proof the main theorem.
\begin{proof}[Proof of \cref{thmB}]
By considering the probability space $([0,1],\mathcal{B}([0,1]),\lambda)$ and $$X_n(y):=\mathbbm{1}_{I_n}(\{u_ny\}),$$
all that remains is to check the hypotheses to use \cref{Lemma_previous_paper}. 

Note that the condition (\ref{Lemma_previous_paper_H1}) follows by \cref{Lemma_estimation_1} and Proposition \ref{Proposition_LLN}. For $b\in (2a,1)$, $M=\lfloor N^{b}\rfloor$ and $y\in [0,1]$, let
$$V_N'(y):=N^{-1+2a-b}\sum_{m=1}^{M}\Big|\sum_{n=1}^{N-m}Y_{n+m}(y)Y_n(y)\Big|.$$
Using Jensen's inequality, it is possible to see that
\begin{equation}
    \begin{split}
        \mathbb{E}(V_N')^2\ll& N^{4a-b-2}\sum_{m=1}^M \sum_{n=1}^{N-m}  \E (Y_n^2Y_{n+m}^2)\\
        &+N^{4a-b-2}\sum_{m=1}^M \sum_{n_2=n_1+1}^{N-m} \mathbb{E}(Y_{n_1}Y_{n_1+m}Y_{n_2}Y_{n_2+m}).\label{ProofA_eq1} \end{split}
\end{equation}
   Since $X_n\in \{0,1\}$, it is not difficult to see that $Y_n^2Y_{n+m}^2\ll X_nX_{n+m}+\sigma_n^2.$
Then, using Lemma \ref{Lemma_sublacunary} and Lemma \ref{Lemma_estimation_1}, there is $\beta>1$ such that
\begin{equation}
    \begin{split}
        &N^{4a-b-2}\sum_{m=1}^M \sum_{n=1}^{N-m}  \E (Y_n^2Y_{n+m}^2)\\
        &\ll N^{4a-b-2} \sum_{m=1}^M\sum_{n=1}^{N-m} \Big(|\Cov(X_n,X_{n+m})|+2n^{-2a}\Big)\\
        &\ll N^{2a-b}+ N^{4a-b-2}\Big(\sum_{m=1}^M \sum_{n=1}^{N}(n+m)^{-a}u_{n}^{-1} + \sum_{m=1}^M\sum_{n=1}^N (n+m)^{-a}u_{n}u_{n+m}^{-1}\Big)\\
        &\ll N^{2a-b}+N^{-1+3a-b}+ N^{4a-b-2}\sum_{n=1}^Nn^{-a}\sum_{m=1}^N \beta^{-d/(n+m)^\delta}\\
        &\ll N^{2a-b}+N^{-1+3a-b}+ N^{4a-b-2}\sum_{n=1}^Nn^{-a}\sum_{m=1}^N \beta^{-d/(2N)^\delta}\\
        &\ll N^{2a-b}+N^{-1+3a-b+\delta},\label{ProofA_eq2}
    \end{split}
\end{equation}
where we used the fact that $(u_n^{-1})_n$ is summable and $$\sum_{n=1}^N\beta^{-n/(2N)^\delta}\leq \dfrac{1}{1-\beta^{-1/(2N)^\delta}}\ll N^\delta.$$
To estimate $\displaystyle \sum_{m=1}^M \sum_{n_2=n_1+1}^{N-m} \mathbb{E}(Y_{n_1}Y_{n_1+m}Y_{n_2}Y_{n_2+m})$, we 
use Lemma \ref{Lemma_estimates_2}. For $k_1<k_2<k_3<k_4$, we have that 
\begin{align}
    \left|\mathbb{E}(Y_{k_1}Y_{k_2}Y_{k_3}Y_{k_4})\right|
\ll(k_3k_2)^{-a}\Big(
\frac{1}{u_{k_1}}
+ \frac{u_{k_1}}{u_{k_2}}
+ \frac{u_{k_2}}{u_{k_3}}
\Big)+k_3^{-a}\frac{u_{k_3}}{u_{k_4}}\Big(k_1^{-a}+\frac{1}{u_{k_1}}+\frac{u_{k_2}}{u_{k_3}}\Big).\label{Extra_final}
\end{align}
For the case $n_1\leq n_1+m\leq n_2\leq n_2+m$, we have that
\begin{equation}
    \begin{split}
        &\sum_{m=1}^M\sum_{n_1=1}^N\sum_{n_2=n_1+m+1}^Nn_2^{-a}(n_1+m)^{-a}\Big(
\frac{1}{u_{n_1}}
+ \frac{u_{n_1}}{u_{n_1+m}}
+ \frac{u_{n_1+m}}{u_{n_2}}
\Big)\\
&\ll N^{2-2a}+\sum_{n_1=1}^N n_1^{-2a}\sum_{m=1}^M\sum_{n_2=n_1+m+1}^N \beta^{-m/(2N)^\delta}+\beta^{-(n_2-n_1-m)/(2N)^\delta}\\
&\ll N^{2-2a}+N^{2-2a+\delta}.\label{ProofA_eq3}
    \end{split}\end{equation}

Similarly,
\begin{equation}
    \begin{split}
      &\sum_{m=1}^M\sum_{n_1=1}^N\sum_{n_2=n_1+m+1}^Nn_2^{-a}\frac{u_{n_2}}{u_{n_2+m}}\Big(n_1^{-a}+\frac{1}{u_{n_1}}+\frac{u_{n_1+m}}{u_{n_2}}\Big)\\
      &\ll N^{2-2a+\delta}+N^{1-a+\delta}+N^{1-a+2\delta}\ll N^{1-2a+2\delta}.\label{ProofA_eq4}
    \end{split}
\end{equation}
Regarding the case $n_1\leq n_2\leq n_1+m\leq n_2+m$, it is easy to see that

\begin{equation}
    \begin{split}
        \sum_{m=1}^N\sum_{n_1=1}^N\sum_{n_2=n_1+1}^{n_1+m}(n_1+m)^{-a}n_2^{-a}\Big(
\frac{1}{u_{n_1}}
+ \frac{u_{n_1}}{u_{n_2}}
+ \frac{u_{n_2}}{u_{n_1+m}}
\Big)\ll N^{2-2a+\delta}.\label{ProofA_eq5}
    \end{split}
\end{equation}
On the other hand, 
\begin{equation}
    \begin{split}
         &\sum_{m=1}^N\sum_{n_1=1}^N\sum_{n_2=n_1+1}^{n_1+m}(n_1+m)^{-a}\frac{u_{n_1+m}}{u_{n_2+m}}\Big(n_1^{-a}+\frac{1}{u_{n_1}}+\frac{u_{n_2}}{u_{n_1+m}}\Big)\\
         &\ll N^{2-2a+\delta}+N^{1-a+\delta}+\sum_{n_1=1}^Nn_1^{-a}\sum_{m=1}^M \sum_{n_2=n_1+1}^{n_1+m}\beta^{\frac{-(n_2-n_1)}{(2N)^\delta}}\beta^{\frac{-(n_1+m-n_2)}{(2N)^\delta}}\\
         &\ll N^{2-2a+\delta}+N^{1-a}\sum_{t=1}^N \beta^{-t/(2N)^\delta}\sum_{m=t}^N\beta^{-(m-t)/(2N)^\delta}\\
         &\ll N^{2-2a+\delta}+N^{1-a+2\delta}.\label{ProofA_eq6}
    \end{split}
\end{equation}

Finally, using the relations \eqref{ProofA_eq2}–\eqref{ProofA_eq6} in \eqref{ProofA_eq1}, it follows that
$$\mathbb{E}(V_N')^2\ll N^{2a-b}+N^{4a-b-2}\Big(N^{2-2a+\delta}+N^{1-a+2\delta}\Big)\ll N^{2a-b}+N^{2a-b+\delta}+N^{-1+3a-b+2\delta}.$$
Since $2a+\delta<1$, we can choose $b$ such that $2a<b<1$, $2a+\delta<b$ and $1<b+a.$ By taking $$\varepsilon:=\frac{1}{2}\min\{b-2a, b-2a-\delta,1+b+a-2(2a+\delta)\},$$
it follows that $$\mathbb{E}(V_N')\ll N^{-\varepsilon}.$$ The proof is completed by invoking \cref{Lemma_previous_paper}.
\end{proof}

\bibliographystyle{abbrv}
\bibliography{bibliography}

\begin{thebibliography}{10}

\bibitem{Austin}
T.~Austin.
\newblock A new dynamical proof of the {S}hmerkin-{W}u theorem.
\newblock {\em J. Mod. Dyn.}, 18:1--11, 2022.

\bibitem{Barlow_Taylor}
M.~T. Barlow and S.~J. Taylor.
\newblock Defining fractal subsets of {$Z^d$}.
\newblock {\em Proceedings of the London Mathematical Society}, s3-64(1):125--152, 1992.

\bibitem{Bellow}
A.~Bellow.
\newblock On ``bad universal''\ sequences in ergodic theory. {II}.
\newblock In {\em Measure theory and its applications ({S}herbrooke, {Q}ue., 1982)}, volume 1033 of {\em Lecture Notes in Math.}, pages 74--78. Springer, Berlin, 1983.

\bibitem{Beresnevich_Ramírez_Velani_2016}
V.~Beresnevich, F.~Ramírez, and S.~Velani.
\newblock {\em Metric Diophantine Approximation: Aspects of Recent Work}, page 1–95.
\newblock London Mathematical Society Lecture Note Series. Cambridge University Press, 2016.

\bibitem{B1}
V.~Bergelson.
\newblock Ergodic {R}amsey theory.
\newblock In {\em Logic and combinatorics ({A}rcata, {C}alif., 1985)}, volume~65 of {\em Contemp. Math.}, pages 63--87. Amer. Math. Soc., Providence, RI, 1987.

\bibitem{Birkhoff_proof_ergodic_theorem:1931}
G.~D. Birkhoff.
\newblock Proof of the ergodic theorem.
\newblock {\em Proc. Natl. Acad. Sci. USA}, 17(12):656--660, 1931.

\bibitem{Bourgain1988}
J.~Bourgain.
\newblock On the maximal ergodic theorem for certain subsets of the integers.
\newblock {\em Israel Journal of Mathematics}, 61:39--72, 1988.

\bibitem{Bourgain_Pointwise_89}
J.~Bourgain.
\newblock Pointwise ergodic theorems for arithmetic sets.
\newblock {\em Inst. Hautes \'Etudes Sci. Publ. Math.}, (69):5--45, 1989.
\newblock With an appendix by the author, Harry Furstenberg, Yitzhak Katznelson and Donald S. Ornstein.

\bibitem{Buczolich_Mauldin_2}
Z.~Buczolich and R.~D. Mauldin.
\newblock Divergent square averages.
\newblock {\em Ann. of Math. (2)}, 171(3):1479--1530, 2010.

\bibitem{DMS}
S.~Donoso, A.~Maass, and V.~Saavedra-Araya.
\newblock A pointwise ergodic theorem along return times of rapidly mixing systems.
\newblock {\em arXiv preprint arXiv:2502.17746}, 2025.

\bibitem{DonosoSun}
S.~Donoso and W.~Sun.
\newblock Pointwise convergence of some multiple ergodic averages.
\newblock {\em Adv. Math.}, 330:946--996, 2018.

\bibitem{Frantzikinakis_Lesigne_Wierdl}
N.~Frantzikinakis, E.~Lesigne, and M.~Wierdl.
\newblock Random sequences and pointwise convergence of multiple ergodic averages.
\newblock {\em Indiana Univ. Math. J.}, 61(2):585--617, 2012.

\bibitem{Frantzikinakis_Lesigne_Wierdl_Szemeredi}
N.~Frantzikinakis, E.~Lesigne, and M.~Wierdl.
\newblock Random differences in {S}zemer\'edi's theorem and related results.
\newblock {\em J. Anal. Math.}, 130:91--133, 2016.

\bibitem{FurstenbergTransversality}
H.~Furstenberg.
\newblock Intersections of {C}antor sets and transversality of semigroups.
\newblock In {\em Problems in analysis ({S}ympos. in honor of {S}alomon {B}ochner, {P}rinceton {U}niv., {P}rinceton, {N}.{J}., 1969)}, pages 41--59. Princeton Univ. Press, Princeton, NJ, 1970.

\bibitem{Furstenberg_ergodic_szemeredi:1977}
H.~Furstenberg.
\newblock Ergodic behavior of diagonal measures and a theorem of {S}zemer\'{e}di on arithmetic progressions.
\newblock {\em J. Anal. Math.}, 31:204--256, 1977.

\bibitem{GMR_transversality}
D.~Glasscock, J.~Moreira, and F.~K. Richter.
\newblock Additive and geometric transversality of fractal sets in the integers.
\newblock {\em J. Lond. Math. Soc. (2)}, 109(5):Paper No. e12902, 55, 2024.

\bibitem{Hauke2025survey}
M.~Hauke.
\newblock A century of metric diophantine approximation and half a decade since {Koukoulopoulos--Maynard}.
\newblock {\em Internationale Mathematische Nachrichten}, 2025.
\newblock survey article, in press.

\bibitem{JONES_LACEY_WIERDL_1999}
R.~L. Jones, M.~Lacey, and M.~Wierdl.
\newblock Integer sequences with big gaps and the pointwise ergodic theorem.
\newblock {\em Ergodic Theory Dynam. Systems}, 19(5):1295--1308, 1999.

\bibitem{Khintchine}
A.~Khintchine.
\newblock Einige {S}\"atze \"uber {K}ettenbr\"uche, mit {A}nwendungen auf die {T}heorie der {D}iophantischen {A}pproximationen.
\newblock {\em Math. Ann.}, 92(1-2):115--125, 1924.

\bibitem{KMPW}
D.~Kosz, M.~Mirek, S.~Peluse, and J.~Wright.
\newblock The multilinear circle method and a question of {B}ergelson.
\newblock {\em arXiv preprint arXiv:2411.09478}, 2024.

\bibitem{Krause_Tao_Mirek}
B.~Krause, M.~Mirek, and T.~Tao.
\newblock Pointwise ergodic theorems for non-conventional bilinear polynomial averages.
\newblock {\em Ann. of Math. (2)}, 195(3):997--1109, 2022.

\bibitem{KrauseSun}
B.~Krause and Y.-C. Sun.
\newblock Quantitative convergence for sparse ergodic averages in {$L^1$}.
\newblock {\em arXiv preprint arXiv:2504.12510}, 2025.

\bibitem{Krause_ZorinKranich_1}
B.~Krause and P.~Zorin-Kranich.
\newblock A random pointwise ergodic theorem with {H}ardy field weights.
\newblock {\em Illinois J. Math.}, 59(3):663--674, 2015.

\bibitem{Krause_ZorinKranich_2}
B.~Krause and P.~Zorin-Kranich.
\newblock A uniform random pointwise ergodic theorem, 2017.
\newblock arXiv preprint arXiv:1708.05022.

\bibitem{LaVictoire}
P.~LaVictoire.
\newblock An {$L^1$} ergodic theorem for sparse random subsequences.
\newblock {\em Math. Res. Lett.}, 16(5):849--859, 2009.

\bibitem{LeVequeI}
W.~J. LeVeque.
\newblock On the frequency of small fractional parts in certain real sequences.
\newblock {\em Trans. Amer. Math. Soc.}, 87:237--261, 1958.

\bibitem{LeVeque}
W.~J. LeVeque.
\newblock On the frequency of small fractional parts in certain real sequences. {III}.
\newblock {\em J. Reine Angew. Math.}, 202:215--220, 1959.

\bibitem{LeVequeII}
W.~J. LeVeque.
\newblock On the frequency of small fractional parts in certain real sequences. {II}.
\newblock {\em Trans. Amer. Math. Soc.}, 94:130--149, 1960.

\bibitem{LeVequeIV}
W.~J. LeVeque.
\newblock On the frequency of small fractional parts in certain real sequences. {IV}.
\newblock {\em Acta Arith.}, 31(3):231--237, 1976.

\bibitem{LIMA_MOREIRA_2014}
Y.~Lima and C.~G. Moreira.
\newblock A {M}arstrand theorem for subsets of integers.
\newblock {\em Combin. Probab. Comput.}, 23(1):116--134, 2014.

\bibitem{Mirek}
M.~Mirek.
\newblock Weak type {$(1,1)$} inequalities for discrete rough maximal functions.
\newblock {\em J. Anal. Math.}, 127:247--281, 2015.

\bibitem{MondalRoyWierdl}
S.~Mondal, M.~Roy, and M.~Wierdl.
\newblock Sublacunary sequences that are strong sweeping out.
\newblock {\em New York J. Math.}, 29:1060--1074, 2023.

\bibitem{saavedraaraya2024distributionintegersdigitrestrictions}
V.~Saavedra-Araya.
\newblock Distribution of integers with digit restrictions via markov chains.
\newblock {\em arXiv preprint arXiv:2411.07418}, 2024.

\bibitem{Schmidt}
W.~M. Schmidt.
\newblock Metrical theorems on fractional parts of sequences.
\newblock {\em Trans. Amer. Math. Soc.}, 110:493--518, 1964.

\bibitem{Shmerkin_transversality}
P.~Shmerkin.
\newblock On {F}urstenberg's intersection conjecture, self-similar measures, and the {$L^q$} norms of convolutions.
\newblock {\em Ann. of Math. (2)}, 189(2):319--391, 2019.

\bibitem{Szusz}
P.~Sz\"usz.
\newblock \"uber die metrische {T}heorie der {D}iophantischen {A}pproximation.
\newblock {\em Acta Math. Acad. Sci. Hungar.}, 9:177--193, 1958.

\bibitem{UrbanZienkiewicz}
R.~Urban and J.~Zienkiewicz.
\newblock Weak type {$(1,1)$} estimates for a class of discrete rough maximal functions.
\newblock {\em Math. Res. Lett.}, 14(2):227--237, 2007.

\bibitem{Wu_transverslity}
M.~Wu.
\newblock A proof of {F}urstenberg's conjecture on the intersections of {$\times p$}- and {$\times q$}-invariant sets.
\newblock {\em Ann. of Math. (2)}, 189(3):707--751, 2019.

\bibitem{HanYu}
H.~Yu.
\newblock On the metric theory of inhomogeneous {D}iophantine approximation: an {Erd\H{o}s}-{V}aaler type result.
\newblock {\em J. Number Theory}, 224:243--273, 2021.

\end{thebibliography}

\end{document}